\documentclass[preprint,11pt]{article}

\usepackage{fullpage}
\usepackage{amssymb,amsfonts,amsmath,amsthm,amscd,dsfont,mathrsfs}
\usepackage{graphicx,float,psfrag,epsfig,amssymb}
\usepackage[usenames,dvipsnames,svgnames,table]{xcolor}
\definecolor{darkgreen}{rgb}{0.0,0,0.9}
\usepackage[pagebackref,letterpaper=true,colorlinks=true,pdfpagemode=none,citecolor=OliveGreen,linkcolor=BrickRed,urlcolor=BrickRed,pdfstartview=FitH]{hyperref}
\usepackage{wrapfig}
\usepackage{relsize}
\usepackage{color}
\usepackage{pict2e}
\usepackage[tight]{subfigure}
\usepackage{algorithm}
\usepackage[noend]{algorithmic}
\usepackage{caption}
\usepackage{nameref}

\DeclareMathAlphabet{\mathpzc}{OT1}{pzc}{m}{it}

\newtheorem{propo}{Proposition}[section]
\newtheorem{lemma}[propo]{Lemma}

\newtheorem{coro}[propo]{Corollary}
\newtheorem{thm}[propo]{Theorem}

\def\hS{\widehat{S}}

\def\cF{{\cal F}}

\def\cC{{\cal C}}
\def\cG{{\cal G}}
\def\cE{{\cal E}}

\def\reals{{\mathbb R}}

\def\eps{{\varepsilon}}
\def\prob{{\mathbb P}}
\def\E{{\mathbb E}}

\def\hOmega {\widehat{\Omega}}

\def\eff{{\rm eff}}
\def\rP{{\rm P}}

\def\L0{{L_0}}

\def\de{{\rm d}}
\def\<{\langle}
\def\>{\rangle}
\def\diag{{\rm diag}}

\def\bX{{\mathbf X}}

\def\htheta{\widehat{\theta}}
\def\hSigma{\widehat{\Sigma}}
\def\hsigma{\widehat{\sigma}}

\def\supp{{\rm supp}}

\def\F{{\sf F}}
\def\ind{{\mathbb I}}

\def\F{{\sf F}}
\def\normal{{\sf N}}

\def\sT{{\sf T}}

\def\id{{\rm I}}

\def\proj{{\rm P}}

\def\cG{{\cal G}}
\def\tZ{\widetilde{Z}}

\def\v*{v_0}
\def\T*{T_0}

\def\u*{u_0}
\def\F*{F_0}

\definecolor{olivegreen}{rgb}{0,0.6,0.4}

\def\hgamma{\widehat{\gamma}}
\def\htau{\widehat{\tau}}

\def\cN{{\cal N}}

\def\utheta{\underline{\theta}}
\def\otheta{\overline{\theta}}

\newcommand{\ajcomment}[1]{}

\makeatletter
\newcommand{\labitem}[2]{%
\def\@itemlabel{\text{#1}}
\item
\def\@currentlabel{#1}\label{#2}}
\makeatother

\addtocontents{toc}{\protect\setcounter{tocdepth}{2}}

\title{Nearly Optimal Sample Size in
Hypothesis Testing for High-Dimensional Regression}

\author{Adel Javanmard
            \footnote{Department of Electrical Engineering, Stanford
              University }
             \,and Andrea~Montanari 
            \footnote{Department of Electrical Engineering and
              Department of Statistics, Stanford University}
            }

\begin{document}

\maketitle

\begin{abstract}
We consider the problem of fitting the parameters of a
high-dimensional linear regression model. In the regime where the
number of parameters $p$ is comparable to or exceeds the sample size $n$,
a successful approach uses an $\ell_1$-penalized least squares
estimator, known as Lasso.

Unfortunately, unlike for linear estimators (e.g., ordinary least
squares), no well-established method exists to compute confidence
intervals or p-values on the basis of the Lasso estimator. Very
recently, a line of work~\cite{javanmard2013hypothesis, confidenceJM, GBR-hypothesis} has addressed this problem by constructing a
debiased version of the Lasso estimator. 
In this paper, we study this approach for random design model, under the assumption that a good estimator
exists for the precision matrix of the design.
Our analysis improves over the state
of the art in that it establishes nearly optimal \emph{average} testing power if the
sample size $n$ asymptotically dominates $s_0 (\log p)^2$, with $s_0$ being the sparsity level (number of non-zero coefficients).
Earlier work obtains provable guarantees only for much larger sample size, namely it requires
$n$ to asymptotically dominate $(s_0 \log p)^2$. 

In particular, for random designs with a sparse precision matrix 
we show that an estimator thereof having the required properties can be computed efficiently.
Finally, we evaluate this approach on synthetic data
and compare it with earlier proposals. 
\end{abstract}

\section{Introduction}

In the random design model for linear regression, we are given $n$ i.i.d. pairs 
$(Y_1,X_1), \dotsc, (Y_n,X_n)$ with $X_i \in \reals^p$.
The response variables $Y_i$ are given by
\begin{eqnarray}\label{eqn:regression}
Y_i \,=\, \<\theta_0,X_i\> + W_i\, ,\;\;\;\;\;\;\;\; W_i\sim
\normal(0,\sigma^2)\, .
\end{eqnarray}
Here $\<\,\cdot\,,\,\cdot\,\>$ is the standard scalar product in
$\reals^p$, and $\theta_0 \in \reals^p$ is an unknown but fixed vector of parameters. 
In matrix form,
letting  $Y = (Y_1,\dotsc,Y_n)^\sT$ and denoting by $\bX$ the design matrix with
rows $X_1^\sT,\dotsc, X_n^\sT$, we have
\begin{eqnarray}\label{eq:NoisyModel}
Y\, =\, \bX\,\theta_0+ W\, ,\;\;\;\;\;\;\;\; W\sim
\normal(0,\sigma^2 \id_{n\times n})\, .
\end{eqnarray}
The goal is to estimate the unknown vector of parameters $\theta_0 \in \reals^p$ from the observations $Y$ and $\bX$.
We are interested in the high-dimensional setting where the number of parameters is larger than the 
sample size, i.e., $p > n$, but the number of non-zero entries of $\theta_0$ is smaller than $p$.
We denote by $S\equiv \supp(\theta_0) \in [p]$ the support of $\theta_0$, i.e., the set of non-zero coefficients, and let $s_0 \equiv |S|$ be the 
sparsity level.

In the last decade, there has been a burgeoning interest in parameter estimation in high-dimensional setting.
A particularly successful approach is the Lasso~\cite{Tibs96,BP95} estimator which
promotes sparse reconstructions through an $\ell_1$ penalty:
\begin{align}
\htheta(Y,\bX;\lambda) \equiv \arg\min_{\theta\in\reals^p}
\Big\{\frac{1}{2n}\|Y-\bX\theta\|^2_2+\lambda\|\theta\|_1\Big\}\, . \label{eq:LassoEstimator}
\end{align}
In case the right hand side has more than one minimizer, one of them
can be selected arbitrarily for our purposes.
We will often omit the arguments $Y$, $\bX$, as they are clear from
the context. 

The Lasso is known to perform well in terms of 
prediction error $\|\bX(\htheta-\theta_0)\|^2_2$ and estimation
error, as measured for instance by $\|\htheta-\theta_0\|_2^2$
\cite{buhlmann2011statistics}.
In this paper we address the --far less understood-- problem of assessing uncertainty and statistical significance,
e.g., by computing confidence intervals or p-values.
This problem is particularly challenging in high dimension since good
estimators, such as the Lasso, are by necessity non-linear and
hence do not have a tractable distribution.

More specifically, we are interested in testing null hypotheses of the form:
\begin{eqnarray}
H_{0,i}: \, \theta_{0,i} = 0\,, \quad \quad \text{ for } i \in [p]\,, \label{eq:hypotheses}
\end{eqnarray}
and assigning $p$-values for these tests. Rejecting $H_{0,i}$
corresponds to inferring that $\theta_{0,i}\neq 0$. 
A related question is the one of computing confidence intervals.
Namely, for a given $i\in[p]$, and $\alpha\in(0,1)$ we want to determine
$\utheta_i,\otheta_i\in \reals$ such that 
\begin{eqnarray}
\prob(\theta_i\in[\utheta_i,\otheta_i])\ge 1-\alpha\, .
\end{eqnarray}

\subsection{Main idea and summary of contributions}

A series of recent papers have developed the idea of `de-biasing' the
Lasso estimator $\htheta$, by defining 
\begin{align}
\htheta^u = \htheta + \frac{1}{n}\, M\bX^\sT(Y- \bX\htheta)\,.
\end{align}
Here $M\in\reals^{p\times p}$ is a matrix that depends on the design matrix $\bX$, and
aims at decorrelating the columns of $\bX$. 
A possible interpretation of this construction is that the term $\bX^\sT(Y- \bX\htheta)/(n\lambda)$
is a subgradient of the $\ell_1$ norm at the Lasso solution $\htheta$. By adding a term
proportional to this subgradient, we compensate for the bias introduced by the $\ell_1$
penalty. It is worth noting that $\htheta^u$ is no longer a sparse
estimator.
In certain regimes, and
for suitable choices of $M$, it was proved that $\htheta^u-\theta_0$
is approximately Gaussian with mean $0$, hence leading to the
construction of p-values and confidence intervals.

More specifically, let $\Sigma = \E(X_1X_1^\sT)$ be the population covariance
matrix, and $\Omega = \Sigma^{-1}$ denote the \emph{precision matrix}. 
In~\cite{javanmard2013hypothesis}, the present authors
assumed the precision matrix to be known and proposed to use $M = c\Omega$,
for an explicit constant $c$. A plug-in estimator for $\Omega$ was also suggested for sparse
covariances $\Sigma$. Furthermore, asymptotic validity and minimax optimality of the method 
were proven for uncorrelated Gaussian
designs ($\Sigma = \id$). A conjecture was derived for a broad class of covariances 
using  statistical physics arguments.
De Geer, B\"uhlmann and Ritov 
\cite{GBR-hypothesis} used  a similar construction with $M$ an estimate of $\Omega$, which is appropriate for
sparse precision matrices $\Omega$.  These authors prove validity of their method
for sample size $n$ that asymptotically dominates $(s_0\log p)^2$. 
In~\cite{confidenceJM}, the present authors propose to construct $M$ by solving a convex program that aims at optimizing
two objectives. First, control the bias of $\htheta^u$, and second minimize the variance of $\htheta^u_i$. 
Minimax optimality was established for sample size $n$ that asymptotically dominates $(s_0\log p)^2$,
without however requiring $\Omega$ to be sparse. Additional related
work can be found in \cite{BuhlmannSignificance,ZhangZhangSignificance} .

Note that nearly optimal estimation via the Lasso is possible for
significantly smaller sample size, namely for $n \ge C s_0 \log
p$, for some constant $C$ \cite{Dantzig,BickelEtAl}.
This suggests the following natural question 
\begin{quote}
\emph{Is it possible to design a minimax optimal test for hypotheses
  $H_{0,i}$, for optimal sample size $n = O (s_0 \log p)$?}
\end{quote}
While the results of \cite{javanmard2013hypothesis} suggest a positive
answer, they assume $\Omega$ to be known, and apply only
asymptotically as $n,p\to\infty$.
In this paper we partially answer this question, by proving the
following results.
\begin{description}
\item[{\bf General subgaussian designs.}] We do not make any
  assumption  on the rows of $\bX$ except that they are independent
  and identically distributed, with common law $p_X$ with subgaussian
  tails, and non-singular covariance. This model is well suited for statistical
  applications wherein the pairs $(Y_i,X_i)$ are drawn at random from
  a population.

Our results in this case holds conditionally on the availability of
  an estimator $\hOmega$ of the precision matrix such
  that\footnote{Here $\|A\|_{\infty}$ denotes the $\ell_{\infty}$
    operator norm of the matrix $A$.}   
$\|\hOmega-\Omega\|_\infty = o(1/\sqrt{\log p})$. Then, a testing procedure
is developed that is minimax optimal with nearly optimal sample size,
namely for $n$ that asymptotically dominates $s_0 (\log p)^2$. Here `optimality' is measured 
in terms of the average power of tests for hypotheses
$H_{0,i}$ with average taken over the coordinates $i\in [p]$. To be more specific, the testing
procedure is constructed based on the debiased estimator $\htheta^u$, where
we set $M= \hOmega$.
\item[{\bf Subgaussian designs with sparse inverse covariance.}] In this case, the rows are subgaussian
  with a common covariance $\Sigma$, such that $\Omega = \Sigma^{-1}$ is sparse. For  this model, an
  estimator with the required properties exists, and was used in
  \cite{GBR-hypothesis}. We can therefore  establish
  unconditional results and prove optimality of the present test. 

   With respect to earlier analysis \cite{GBR-hypothesis}, our results
   apply to much smaller sample size, namely $n$ needs to dominate
   $s_0(\log p)^2$ instead of $(s_0\log p)^2$. On the other hand,
   guarantees are only provided with respect to average  power of the
   test.
   Roughly speaking, our results in this case imply that the method of
   \cite{GBR-hypothesis}   has significantly broader domain of validity than
   initially expected.
\end{description}
While the assumption of sparse inverse covariance
is admittedly restrictive, it arises naturally in a number of contexts.
For instance, it is relevant for the problem of learning sparse
Gaussian graphical models \cite{MeinshausenBuhlmann}. In this case, the set
of edges incident on a specific vertex can be encoded in the vector
$\theta_0$. It also played a pivotal role in compressed sensing, as
one of the first model in which an optimal tradeoff between sparsity $s_0$
and sample size $n$ was proven to hold
\cite{CandesTao,DoTa05,Wainwright2009LASSO}.

Covariance estimators satisfying the condition
$\|\hOmega-\Omega\|_\infty = o(1/\sqrt{\log p})$ can be constructed
under other structural assumptions than sparsity. Our general 
theory allows to build hypothesis testing methods for each of these cases.
We expect this to spur progress in other settings as well.

Finally, we evaluate our procedure on synthetic data, comparing its performance with the 
method of~\cite{confidenceJM}. 

\subsection{Definitions and notations}

Throughout $\Sigma = \E\{X_1X_1^{\sT}\}$ will be referred to as the
covariance, and $\Omega\equiv \Sigma^{-1} \in \reals^{p\times p}$ as
the precision matrix.
Without loss of generality, we will assume that the columns of $\bX$ are normalized so that
$\Sigma_{ii}=1$.  (This normalization is only assumed for the
analysis, and is not required for  the hypothesis testing procedure or the
construction of confidence intervals.)

For a matrix $A$ and set of indices $I,J$, we let $A_{I,J}$ denote the submatrix formed 
by the rows in $I$ and columns in $J$. Also,
$A_{I,\cdot}$ (resp. $A_{\cdot,I}$) denotes the submatrix
containing just the rows (resp. columns) in $I$. 
Likewise, for a vector $v$, $v_I$ is the restriction
of $v$ to indices in $I$. 
The maximum and the minimum singular values of $A$ are respectively denoted 
by $\sigma_{\max}(A)$ and $\sigma_{\min}(A)$.
We write $\|v\|_p$ for the standard $\ell_p$ norm of a vector $v$
(omitting the subscript in the case $p=2$)
and  $\|v\|_0$ for the number of nonzero entries of  $v$. 
For a matrix $A$, $\|A\|_p$ is its $\ell_p$ operator norm, and $|A|_p$ is the elementwise $\ell_p$ norm, i.e., $|A|_p = (\sum_{i,j} |A_{ij}|^p)^{1/p}$.
Further $|A|_\infty = \max_{i,j}|A_{ij}|$.
We use the notation $[n]$ for the set $\{1,\dotsc,n\}$. 
For a vector $v$, $\supp(v)$ represents
the positions of nonzero entries of $v$. 

The standard normal distribution function is denoted by
$\Phi(x) \equiv \int_{-\infty}^x e^{-t^2/2} \de t/\sqrt{2\pi}$. For two functions $f(n)$ and $g(n)$,  the notation
$f(n) = \omega(g(n))$ means that $f$ dominates $g$ asymptotically, namely, for
every fixed positive $C$, there exists $n_0$ such that $f(n) \ge C g(n)$ for $n > n_0$.

The sub-gaussian norm of a random variable $Z$, denoted by $\|Z\|_{\psi_2}$, is 
defined as
$$\|Z\|_{\psi_2} = \sup_{q\ge 1} q^{-1/2} (\E|Z|^q)^{1/q}\,.$$
The sub-gaussian norm of a random vector $Z$ is defined as
$\|Z\|_{\psi_2} = \sup_{\|x\|=1} \|\<Z,x\>\|_{\psi_2}$.

Finally, the sub-exponential norm of random variable $Z$ is defined as
$$\|Z\|_{\psi_1} = \sup_{q\ge 1} q^{-1} (\E|Z|^q)^{1/q}\,.$$

\section{Debiasing the Lasso estimator}
\label{sec:Debiased}
Let $\hOmega$ be an estimate of the precision matrix $\Omega$.
We define estimator $\htheta^u$ based on the Lasso solution $\htheta$ and
 $\hOmega$, as per Eq.~\eqref{eq:hthetau} in Table $1$. The following proposition provides 
 a decomposition of the residual $\htheta^u- \theta_0$, which is useful in characterizing the limiting distribution
 of $\htheta^u$. Its proof follows readily from the proof of Theorem 2.1 in~\cite{GBR-hypothesis}, and
 is given in Appendix~\ref{app:main_thm} for the reader's convenience.
\begin{propo}\label{pro:main_thm}
Consider the linear model~\eqref{eqn:regression} and let $\htheta^u$ be defined as per
Eq.~\eqref{eq:hthetau}.
Then,
\begin{align*}
\sqrt{n} (\htheta^u - \theta_0) &= Z + \Delta\,, \\
Z | \bX \sim \normal(0,\sigma^2 \hOmega \hSigma\hOmega^\sT)\,, \quad&
\Delta = \sqrt{n} (\hOmega \hSigma-\id) (\theta_0 - \htheta)\,.
\end{align*}
\end{propo}

\begin{algorithm}[t]
\caption*{{\bf Table 1:} Unbiased estimator for $\theta_0$ in high dimensional linear regression models}
\begin{algorithmic}[1]

\REQUIRE Measurement vector $y$, design matrix $\bX$, parameter $\lambda_n$, estimated precision matrix $\hOmega$.

\ENSURE Unbiased estimator $\htheta^u$.

\STATE Let $\htheta = \htheta(\lambda_n)$ be the Lasso estimator as per Eq.~\eqref{eq:LassoEstimator}.

\STATE Define the estimator $\htheta^u$ as follows:
\begin{eqnarray}
\htheta^u = \htheta + \frac{1}{n}\, \hOmega \bX^\sT(Y - \bX \htheta) \label{eq:hthetau}
\end{eqnarray}
\end{algorithmic}
\end{algorithm}

We recall the definition of \emph{restricted eigenvalues} as given in~\cite{BickelEtAl}:
$$ \phi_{\max}(t) \equiv \underset{1\le \|v\|_0 \le t}{\max}\, \frac{\|\bX v\|_2^2}{n\|v\|_2^2}\,.$$
It is also convenient to recall the following \emph{restricted eigenvalue} (RE) assumptions. 
\begin{description}
\item[{\sc Assumption} ${\rm RE}(s,c)$.]
For some integer $s$ such that $1\le s \le p$ and a positive number $c$, the following condition holds:
\[
\kappa(s,c) \equiv \min_{J\subseteq [p]: |J|\le s} \quad \min_{v \neq 0: \|v_{J^c}\|_1 \le c \|v_J\|_1} \quad \frac{\|\bX v\|_2}{\sqrt{n} \|v_J\|_2} > 0\,.
\]
The assumption RE$(s,c)$ has been used to establish bounds on the prediction loss 
and on the $\ell_1$ loss of the Lasso. 
\item[{\sc Assumption} ${\rm RE}(s,q,c)$.]
 Let $s, q$ be integers such that $1\le s \le p/2$ and $q \ge s$, $s+q \le p$.
For a vector $v \in \reals^p$ and a set of indices $J \subseteq [p]$ with $|J| \le s$, denote by $J_1$ the subset of 
$[p]$ corresponding to the $q$ largest coordinates of $v$ (in absolute value) and define $J_2 \equiv J \cup J_1$.
We say that $\bX$ satisfies ${\rm RE}(s,q,c)$ with constant
$\kappa(s,q,c)$ if 
\[
\kappa(s,q,c) \equiv \min_{J\subseteq[p]: |J| \le s} \quad \min_{v\neq0: \|v_{J^c}\|_1 \le c \|v_J\|_1} \quad \frac{\|\bX v\|_2}{\sqrt{n} \|v_{J_2}\|_2}>0\,.
\]
This  assumption has been used to bound the $\ell_p$
loss of the Lasso with $1<p \le 2$~~\cite{BickelEtAl}.
\end{description}

The following lemma is a minor improvement over \cite[Theorem 7.2]{BickelEtAl} in that it uses $\phi_{\max}(n)$
instead of $\phi_{\max}(p)$.
\begin{propo}[\cite{BickelEtAl}]\label{pro:Bickel}
Let  assumption RE$(s_0,3) > 0$ be satisfied. Consider the Lasso selector $\htheta$ with $\lambda = \sigma \sqrt{2\log p/n}$.
Then, with high probability, we have
\begin{eqnarray}
\|\htheta\|_0 &\le& \frac{64\phi_{\max}(n)^2}{\kappa(s_0,3)^2} s_0\,.\label{eq:Lasso-supp-size}
\end{eqnarray}
If assumption RE$(s_0,q,3)$ with constant $\kappa = \kappa(s_0,q,3)$ is satisfied, then with high probability,
\begin{eqnarray}
\|\htheta - \theta_0\|_2^2 \le Cs_0 \frac{\sigma^2 \log p}{n}\,, \label{eq:l2}
\end{eqnarray}
where $C = C(\kappa)$ is bounded for $\kappa$ bounded away from $0$.
\end{propo}
A proof of Eq.~\eqref{eq:Lasso-supp-size} is given in Appendix~\ref{app:Lasso-supp-size}.

Our next theorem controls the bias term $\Delta$.
In order to state the result formally, for a vector $v\in\reals^m$, and
$k\le m$, we
define its $(\infty,k)$ norm as follows
\begin{align}
\|v\|_{(\infty,k)} \equiv \max_{A\subseteq [m], |A|\ge k}
\frac{\|v_A\|_2}{\sqrt{k}}\, .
\end{align}
For $k=1$, this is just the $\ell_{\infty}$ norm (the maximum entry)
of $v$. At the other extreme, for $k=m$, this is the rescaled $\ell_2$
norm. It is easy to see that $\|v\|_{(\infty,k)}$ is non-increasing in
$k$. As $k$ gets smaller, it gives us tighter control on
the individual entries of $v$.

The next theorem bounds $\|\Delta\|_{(\infty,k)}$ down to $k$ much
smaller than $s_0$.
\begin{thm}\label{thm:main_lem}
Consider the linear model~\eqref{eqn:regression} and let $\Sigma$ be the population covariance
matrix of the design $\bX$. Let $\Omega \equiv \Sigma^{-1}$ be the precision matrix and
suppose that an estimate $\hOmega$ is available, such that
$\|\hOmega- \Omega\|_\infty = o_\rP(1/\sqrt{\log p})$. 
Further, assume that $\sigma_{\min}(\Sigma)$ and $\sigma_{\max}(\Sigma)$ are respectively bounded from below
and above by some constants as $n\to \infty$. In addition, assume that
the rows of the whitened matrix $\bX \Omega^{1/2}$ are 
sub-gaussian, i.e., $\|\Omega^{1/2} X_1\|_{\psi_2} < C_1$, for some constant $C_1>0$. 

Let $\Delta \equiv \sqrt{n} (\hOmega \hSigma-\id) (\theta_0 -
\htheta)$ be the bias term in $\htheta^u$. Then for
any arbitrary (but fixed) constant $c >0$, there exists
$C = C(c, C_1,\sigma_{\max}(\Sigma),\sigma_{\min}(\Sigma))<\infty$
such that, 
\begin{eqnarray}\label{eq:Delta}
\|\Delta\|_{(\infty,cs_0)}^2 \le C\frac{\sigma^2 s_0 (\log p)^2}{n} + o_\rP(1)\,.
\end{eqnarray}
%
%
\end{thm}
The proof is deferred to Section~\ref{proof:main_lem}.

Using Markov inequality, this implies that there cannot be many entries
of $\Delta$ that are large.
\begin{coro}\label{cor:main}
Under the conditions of Theorem~\ref{thm:main_lem}, and for $n = \omega(s_0(\log p)^2)$, there are at most
$o(s_0)$ entries of $\Delta$ that are of order $\Omega(1)$. More
precisely, fix arbitrary $\eps >0$,
and define 
$\cC_n(\eps) \equiv \{i\in[p]:\, |\Delta_i| > \eps\}$. Then the
following limit holds in probability
\[
\underset{n\to \infty}{\lim} \frac{1}{s_0} |\cC_n(\eps)| = 0\,. 
\]
\end{coro}
\begin{proof}
If the claim does not hold true, then by applying Theorem~\ref{thm:main_lem} to the set $\cC_n(\eps)$, we 
have
\[
\eps^2 \le \frac{\|\Delta_{\cC_n(\eps)}\|_2^2}{|\cC_n(\eps)|} \le C \frac{\sigma^2 s_0 (\log p)^2}{n} + o_\rP(1) = o_\rP(1)\,,
\]
which is contradiction.
\end{proof}
In other words, except for at most $o(s_0)$ entries of $\theta_0$, $\htheta^u_i$
is an asymptotically unbiased estimator for $\theta_{0,i}$. 
%
%
\section{Constructing $p$-values and hypothesis testing}
For the linear model~\eqref{eqn:regression}, we are interested in testing the individual hypotheses
$H_{0,i}:\, \theta_{0,i} = 0$, and assigning $p$-values
for these tests. 

Similar to~\cite{confidenceJM}, we construct a $p$-value $P_i$ for the test $H_{0,i}$ as follows:
\begin{eqnarray}\label{eq:p-value}
P_i = 2\Big(1-\Phi\Big(\frac{\sqrt{n} |\htheta^u_i|}{\hsigma [\hOmega \hSigma \hOmega^\sT]_{ii}^{1/2}}\Big) \Big)\,.
\end{eqnarray}
where $\hsigma$ is a consistent estimator of $\sigma$. 
For instance, we can use the scaled Lasso~\cite{SZ-scaledLasso}
(see also related work in \cite{belloni2011square}) as
\[
\{\htheta, \hsigma\} \equiv \underset{\theta\in \reals^p, \sigma >0}{\arg\min} 
\Big\{\frac{1}{2\sigma n} \|Y- \bX \theta\|_2^2 + \frac{\sigma}{2} + \lambda \|\theta\|_1 \Big\}\,.
\]
Choosing $\lambda = O(\sqrt{(\log p)/n})$ yields a consistent estimate
$\hsigma$ of $\sigma$.

A different estimator of $\sigma$ can be constructed using Proposition
\ref{pro:main_thm} and Theorem \ref{thm:main_lem}, and is described in
Section~\ref{sec:SigmaEstimator}.

The decision rule is then based on the $p$-value $P_i$:
\begin{eqnarray}
\begin{split}\label{eq:decision-rule}
T_{i,\bX}(y) = \begin{cases}
1 & \text{if } P_i \le \alpha \quad \quad \text{ (reject $H_{0,i}$)}\,,\\
0 & \text{otherwise} \quad\quad \text{(accept $H_{0,i}$)}\,.
\end{cases}
\end{split}
\end{eqnarray} 
We measure the quality of the test $T_{i,\bX}(y)$ in terms of its
significance level $\alpha_i$ 
 and statistical power $1-\beta_i$. Here $\alpha_i$ is the probability
 of type I error (i.e., of a false positive at $i$)
 and $\beta_i$ is the probability of type II error (i.e., of a false
 negative at $i$).

Our next theorem characterizes the tradeoff between type I error and the average power attained by the decision rule~\eqref{eq:decision-rule}.
 Note that this tradeoff depends on the magnitude
of the non-zero coefficients $\theta_{0,i}$. The larger they are, the easier one can
distinguish between null hypotheses from their alternatives.
 We refer to Section~\ref{proof:error-power} for a proof.
\begin{thm}\label{thm:error-power}
Consider a random design model that satisfies the conditions of Theorem~\ref{thm:main_lem}.
For $\theta_0 \in \reals^p$, let $S \equiv \{i\in [p]:\, \theta_{0,i}
\neq 0\}$. Assume that $\hOmega$ is such that
$\|\hOmega-\Omega\|_{\infty} =o(1/\sqrt{\log p})$ with high probability.
Under the sample size assumption $n = \omega (s_0 (\log p)^2)$, the following holds true:
\begin{align}
&\underset{n \to \infty}{\lim\sup}\, \frac{1}{p - s_0} \sum_{i\in S^c} \prob_{\theta_0} (T_{i,\bX}(y)=1) \le \alpha\,.\label{eq:typeI}\\
&\underset{n \to \infty}{\lim\inf}\, \frac{1}{1-\beta^*(\theta_0;n)} \Big\{ \frac{1}{s_0} \sum_{i\in S} \prob_{\theta_0} (T_{i,\bX}(y)=1) \Big\} \ge 1\,, \label{eq:power}\\
& 1- \beta^*(\theta_0;n) \equiv \frac{1}{s_0} \sum_{i\in S} G\bigg(\alpha,\frac{\sqrt{n}\,|\theta_{0,i}|}{\sigma \sqrt{\Omega_{ii}}}\bigg)\,, \label{eq:beta}
\end{align}
where, for $\alpha \in [0,1]$ and $u \in \reals_+$, the function $G(\alpha,u)$ is defined as follows:
\begin{eqnarray}\label{Eq:Gfun}
G(\alpha,u) = 2 - \Phi(\Phi^{-1}(1-\frac{\alpha}{2}) + u) - \Phi(\Phi^{-1}(1-\frac{\alpha}{2}) - u)\,. 
\end{eqnarray}
Furthermore, $\prob_{\theta_0}(\cdot)$ is the induced probability for random design $\bX$ and noise
realization $w$, given the fixed parameter vector $\theta_0$.
\end{thm}

In Fig.~\ref{fig:G}, function $G(\alpha, u)$ is plotted versus $\alpha$, for several values of $u$.
It is easy to see that, for any $\alpha \in (0,1)$, $u \mapsto G(\alpha,u)$ is monotone increasing. 
Suppose that $\min_{i\in S} |\theta_{0,i}| \ge \mu$.  Then, by Eq.~\eqref{eq:beta}, we have
$$ 1-\beta^*(\theta_0,n) \ge \frac{1}{s_0} \sum_{i\in S} G\Big(\alpha, \frac{\sqrt{n} \mu}{\sigma \sqrt{\Omega_{ii}}}\Big)\,.$$
Notice that $G(\alpha,0) = \alpha$, giving the lower bound $\alpha$ for the power at $\mu = 0$.
In fact without any assumption on the non-zero coordinates of $\theta_0$, one can take $\theta_{0,i} \neq 0$ arbitrarily
close to zero, and practically $H_{0,i}$ becomes indistinguishable from its alternative. 
In this case, no decision rule can outperform random guessing, i.e, randomly rejecting $H_{0,i}$ with probability $\alpha$.
This yields the trivial power $\alpha$.

\begin{figure}
\centering
\includegraphics[viewport = 0 160 600 640, width = 4.2in]{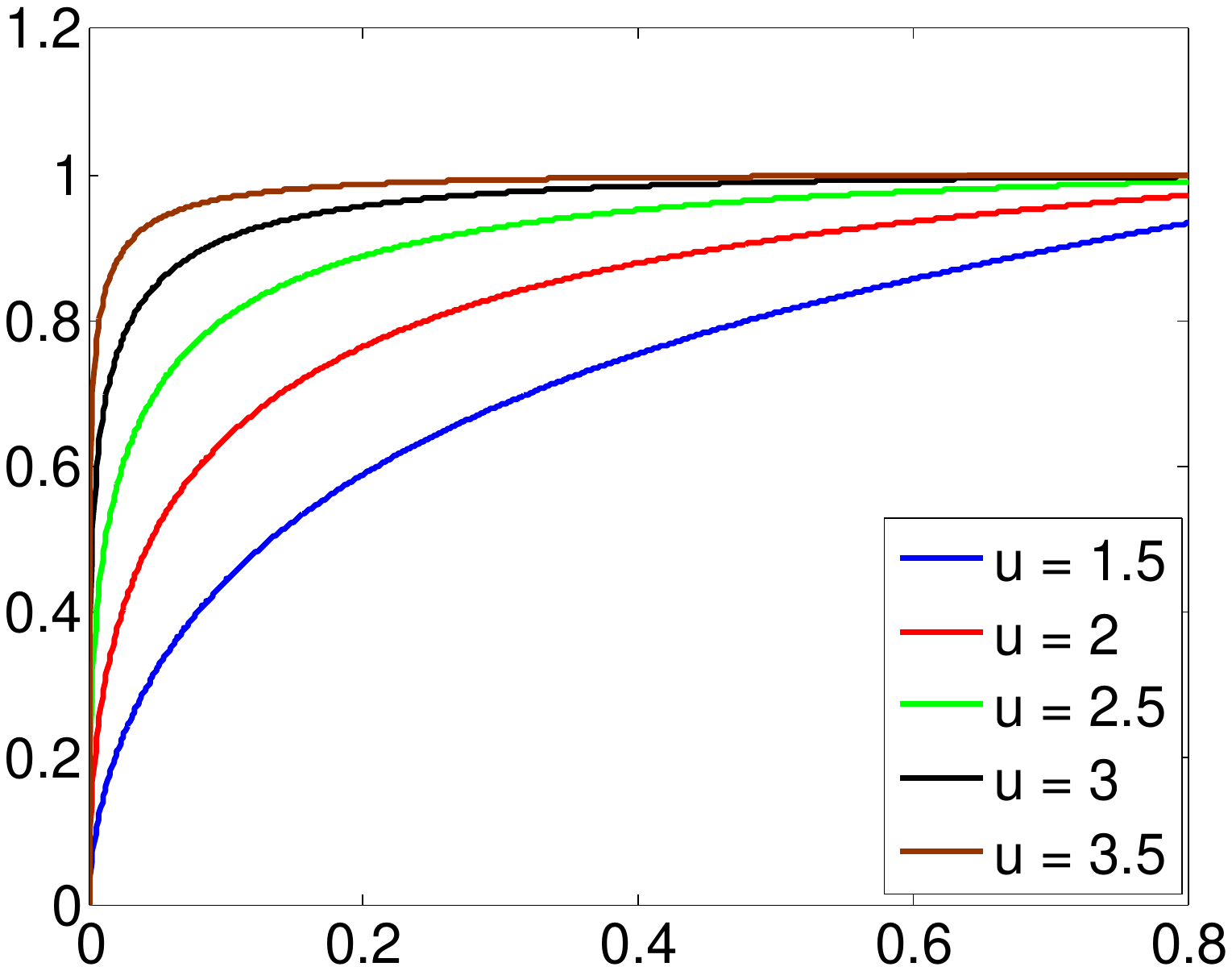}
\put(-170,10){{\small{\sf Type I error }($\alpha$)}}
\put(-275,100){\rotatebox{90}{$G(\alpha,u)$}}
\caption{Function $G(\alpha,u)$ versus $\alpha$ for several values of $u$.}\label{fig:G}
\end{figure}
%
\subsection{Minimax optimality of the average power}
An upper bound for the minimax power of tests, with a given significant level $\alpha$,
is provided in~\cite{javanmard2013hypothesis} for sparse linear regression with Gaussian designs.
Considering sample size scaling $n =\omega(s_0 (\log p)^2)$, 
the minimax bound~\cite[Theorem 2.6]{javanmard2013hypothesis} simplifies to the following bound for the optimal average
power
\begin{eqnarray}
\lim_{n\to \infty} \frac{1-\beta^{{\rm
      opt}}(\alpha;\mu)}{G(\alpha,\mu/\sigma_\eff)}\le 1\, ,\quad
\sigma_\eff = \frac{\sigma}{\sqrt{n \eta_{\Sigma,s_0}}}\,, \label{eq:optimal-power}
\end{eqnarray}
where 
\begin{align*}
\eta_{\Sigma,s_0} &\equiv \max_{i\in[p]}\min_{S} \Big\{\Sigma_{i|S}:\, S\subseteq[p]\backslash \{i\}, |S| < s_0 \Big\}\,,\\
\Sigma_{i|S} &\equiv \Sigma_{ii} - \Sigma_{i,S}(\Sigma_{S,S})^{-1}\Sigma_{S,i}\,.
\end{align*}
We compare our test to the optimal test by computing how much $\mu$ must be increased to
achieve the minimax optimal average power. It follows from Eqs.~\eqref{eq:power} and~\eqref{eq:beta} that
$\mu$ must be increased to $\tilde{\mu}$, with the increase factor 
\[
\frac{\tilde{\mu}}{\mu} = \Big(\max_{i\in [p]} \sqrt{\Omega_{ii}}\Big) \sqrt{\eta_{\Sigma,s_0}}
\le \max_{i\in [p]} \sqrt{\Omega_{ii} \Sigma_{ii}} \le \sqrt{\frac{\sigma_{\max}(\Sigma)}{\sigma_{\min}(\Sigma)}}\,.
\]

Therefore, our test has nearly optimal average power for
well-conditioned covariances and sample size scaling $n = \omega(s_0 (\log p)^2)$.


\section{An estimator of the noise level}
\label{sec:SigmaEstimator}
In this section we describe a consistent estimator $\hsigma$ of the
noise standard deviation $\sigma$, that is based on Proposition
\ref{pro:main_thm} and Theorem \ref{thm:main_lem}.
After constructing $\htheta^{u}$ as per Eq.~\eqref{eq:hthetau}, we
let
\begin{align}
z_i\equiv \frac{\sqrt{n}\,
  \htheta^u_i}{[\hOmega\hSigma\hOmega^\sT]^{1/2}_{ii}}\, .
\end{align}
According to Proposition
\ref{pro:main_thm} and Theorem \ref{thm:main_lem}, the entries $z_i$,
$i\not\in S$ are approximately Gaussian, with mean $0$ and variance
$\sigma^2$. This suggests to use the following robust estimator (a
similar approach in a related context was proposed in \cite{DMM09}). 

Let $|z|$ be the vector of absolute values of $z$, i.e. 
$|z| = (|z_1|,|z_2|,\dots,|z_p|)$, and denote by $|z|_{(i)}$ its
$i$-th entry in order of magnitude: $|z|_{(1)}\le |z|_{(2)}\le \dots
\le |z|_{(p)}$. We then set\footnote{More generally, for $\alpha\in
  (0,1)$, we can use $\hsigma_{\alpha} \equiv |z|_{(p\alpha)}/\Phi^{-1}((1+\alpha)/2)$.}
\begin{align}
\hsigma = \frac{|z|_{(p/2)}}{\Phi^{-1}(3/4)}
\end{align}
The next corollary is an immediate consequence of Proposition
\ref{pro:main_thm} and Theorem \ref{thm:main_lem} (see also Lemma \ref{lem:BoundVariance}).
\begin{coro}
Under the assumptions of Theorem \ref{thm:main_lem}, if further
$s_0= o(p)$, we have $\hsigma\to \sigma$ in probability.
\end{coro} 


\section{Designs with sparse precision matrix}

In Theorem~\ref{thm:main_lem} we posit existence of an estimator
$\hOmega$ for the precision matrix $\Omega$, such that $\|\hOmega - \Omega\|_\infty = o(1/\sqrt{\log p})$. 
In case the precision matrix is sparse enough, then~\cite{GBR-hypothesis} constructs such an estimator using 
the Lasso for the nodewise regression on the design $\bX$. Formally, for $j \in [p]$, let
$$\hat{\gamma}_j = \underset{\gamma}{\arg\min}\, \frac{1}{2n} \|X_j - \bX_{-j}\gamma\|_2^2 + \lambda \|\gamma\|_1\,,$$  
where $\bX_{-j}$ is the sub-matrix obtained by removing the $j^{th}$ column. Also let
$$ \widehat{C} = 
\begin{bmatrix}
1 & -\hgamma_{1,2} & \cdots & -\hgamma_{1,p}\\
-\hgamma_{2,1} & 1 & \cdots & -\hgamma_{1,p}\\
\vdots & \vdots & \ddots & \vdots\\
-\hgamma_{p,1} & -\hgamma_{p,2} & \cdots &1  
\end{bmatrix}\,,
$$
with $\hgamma_{j,k}$ being the $k$-th entry of $\hgamma_j$, and let  
$$ \widehat{T}^2 = \diag(\htau_1^2,\dotsc,\htau_p^2),\quad \htau_j^2 = (X_j - \bX_{-j} \hgamma_j)^\sT X_j/n\,.$$ 
Then define $\hOmega = \widehat{T}^{-2} \widehat{C}$.

\begin{propo}[\cite{GBR-hypothesis}]
\label{propo:GBR}
Suppose that the whitened matrix $\bX \Omega^{1/2}$ has i.i.d. sub-gaussian rows. Further, assume that
the maximum number of non-zeros per row of $\Omega$ is $t_0 = o(n/\log p)$,
that  $\sigma_{\max}(\Omega) = \sigma_{\rm min}(\Sigma)^{-1}=O(1)$ and
that, for all $i\in [p]$, $\Sigma_{ii}=1$. Then, with high probability,
\begin{eqnarray}
\|\hOmega - \Omega\|_\infty = O(t_0 \sqrt{\log p /n})\,. \label{eq:spectral-loss}
\end{eqnarray}
\end{propo}

Therefore, if $\Omega$ is sufficiently sparse, namely it has $t_0 = o(\sqrt{n}/\log p)$ non-zeros per row, 
then Eq.~\eqref{eq:spectral-loss} yields $\|\hOmega - \Omega\|_\infty = o(1/\sqrt{\log p})$.
Hence, the estimator $\hOmega$ satisfies the 
conditions of  Theorem~\ref{thm:main_lem} and \ref{thm:error-power}.

\section{Numerical experiments}
\label{sec:experiment}
We  generated synthetic data from the linear
model~\eqref{eqn:regression} with the choice of parameters
$\sigma =1$, $n = 240$ and $p = 300$. The rows of the design matrix $\bX$ are generated
independently form distribution $\normal(0, \Sigma)$. Here $\Omega = \Sigma^{-1}$ is a circulant matrix with
$\Omega_{ii} = 1$,  $\Omega_{jk} = a$ for $j \neq k$, $|j-k| \le b$, and zero everywhere else.
(The difference between indices is understood modulo $p$.) The parameter $b$ controls the sparsity of
the precision matrix and we take $a = 1/b$ to ensure that $\Omega \succ 0$ (with $\sigma_{\min}(\Omega) > 0.5$).
For parameter vector $\theta_0$, we consider a subset $S \subseteq [p]$, with $|S| = s_0 = 30$, chosen 
 uniformly at random, and set $\theta_{0,i} = 0.1$ for $i \in S$ and zero everywhere else. 

We evaluate the performance of our testing procedure~\eqref{eq:decision-rule} at significance level $\alpha = 0.05$.
The procedure is implemented in $R$ using {\sf glmnet}-package that fits the Lasso solution for an entire path
of regularization parameters $\lambda_n$. We then choose the value of $\lambda_n$ that has the minimum 
mean squares error, approximated by a $5$-fold cross validation.

We compare the performance of decision rule~\eqref{eq:decision-rule} to the testing method presented in~\cite{confidenceJM} for different values of
$b$. (Recall that $b$ controls the sparsity of the precision matrix.) The results are reported in Table~\ref{tbl:comparison}.
The means and the standard deviations are obtained by testing over $20$ realizations of noise and the design matrix.

Interestingly, for small values of $b$ (very sparse precision
matrices), the two methods
perform often identically the same, and their performances differ slightly for moderate $b$. This is in agreement with the theoretical
results  that both methods asymptotically have nearly optimal minimax average power.

Letting $Z = (z_i)_{i=1}^p$ with $z_i \equiv \sqrt{n}(\htheta^u_i - \theta_{0,i})/ [\hOmega \hSigma \hOmega]_{ii}^{1/2}$,
in Fig.~\ref{fig:qqplot} we plot sample quantiles of $Z$ versus the quantiles of a standard normal distribution for one realization
(with $b = 75$). The linear trend of the plot clearly demonstrates that the empirical distribution of $Z$ is approximately
normal, corroborating our theoretical results (Proposition~\ref{pro:main_thm} and Theorem~\ref{thm:main_lem}).

\begin{table*}[]
\begin{center}
{\small
\begin{tabular}{|c|cccc|}
\hline
{\bf Method} & Type I err & Type I err & Avg. power & Avg. power \\ 
&(mean) & (std.) & (mean) & (std)
\\
\cline{1-5} \cline{2-5}
\multicolumn{1}{|c|}{\bf Present testing procedure  $(b = 5)$} &  0.0644 & 0.0060 & 0.5766 & 0.0387
 \\ 
 \multicolumn{1}{|c|}{\bf Procedure of~\cite{confidenceJM} $(b = 5)$} & \color{red} 0.0644 & \color{red} 0.0060 & \color{red} 0.5766& \color{red} 0.0387
 \\
 \hline
\multicolumn{1}{|c|}{\bf Present testing procedure  $(b = 25)$} &  0.0600  & 0.0074 & 0.5750 & 0.0445
 \\ 
\multicolumn{1}{|c|}{\bf Procedure of~\cite{confidenceJM} $(b = 25)$} &  \color{red} 0.0600 & \color{red} 0.0074 & \color{red} 0.5750 & \color{red} 0.0445
\\
\hline 
\multicolumn{1}{|c|}{\bf Present testing procedure  $(b = 50)$} &0.0412   & 0.0061 &0.5350  & 0.0383
 \\ 
\multicolumn{1}{|c|}{\bf Procedure of~\cite{confidenceJM} $(b = 50)$} &  \color{red} 0.0468 & \color{red} 0.0063  & \color{red} 0.5416 & \color{red} 0.0386
\\
\hline 
\multicolumn{1}{|c|}{\bf Present testing procedure  $(b = 75)$} & 0.0509   & 0.0075 &  0.4916 & 0.0334
 \\ 
\multicolumn{1}{|c|}{\bf Procedure of~\cite{confidenceJM} $(b = 75)$} &  \color{red} 0.0507 & \color{red} 0.0073 & \color{red}0.4900  & \color{red} 0.0340
\\
\hline 
\multicolumn{1}{|c|}{\bf Present testing procedure  $(b = 100)$} & 0.0479 & 0.0067 & 0.5150 & 0.0310
 \\ 
\multicolumn{1}{|c|}{\bf Procedure of~\cite{confidenceJM} $(b = 100)$} &  \color{red} 0.0618 & \color{red} 0.0077 & \color{red}0.5416 & \color{red} 0.0302
\\
\hline
\end{tabular}
}
\end{center}
\caption{Comparison between testing procedure \eqref{eq:decision-rule} and procedure proposed in~\cite{confidenceJM} on the setup described in Section~\ref{sec:experiment}. The significance level is $\alpha = 0.05$. The means and the standard deviations are obtained by testing over $20$ realizations.}\label{tbl:comparison}
\end{table*}

\section{Proofs}\label{sec:proofs}
\subsection{Proof of Theorem~\ref{thm:main_lem}}\label{proof:main_lem}

Write $\Delta  = \Delta^{(1)} + \Delta^{(2)}$ with
\[
\Delta^{(1)} = \sqrt{n} (\Omega \hSigma - \id)(\theta_0 - \htheta)\,,
\quad
\Delta^{(2)} = \sqrt{n} (\hOmega - \Omega) \hSigma(\theta_0 - \htheta)\,.
\]

Let $T\equiv \supp(\htheta)\cup\supp(\theta_0)$. By Eq.~\eqref{eq:Lasso-supp-size},
$|T| = O(s_0)$ because $\phi_{\max}(n)\le C$, $\kappa(s_0,3) \ge  1/C$
for some constant $C<\infty$, with
high probability.
(This in turns follows from the assumption that $\sigma_{\rm
  max}(\Sigma)$, $\sigma_{\rm min}(\Sigma)$ are bounded above and
below, using \cite{rudelson2011reconstruction}.)
Also, note that any set $A\subseteq [p]$ with $|A| \ge cs_0$ can be partitioned as $A = \cup_{\ell=1}^L
A_{\ell}$ with $c\, s_0\le |A_{\ell}|\le 2c\, s_0$. If the claim holds for all
$A_{\ell}$, then it follows for $A$ by summing these cases. We can therefore
assume, without loss of generality, $c\, s_0\le |A|\le 2c\, s_0$.
\begin{figure}[!t]
\centering
\subfigure[$b = 5$] {
\includegraphics*[viewport = 15 0 480 455, width =
2.9in]{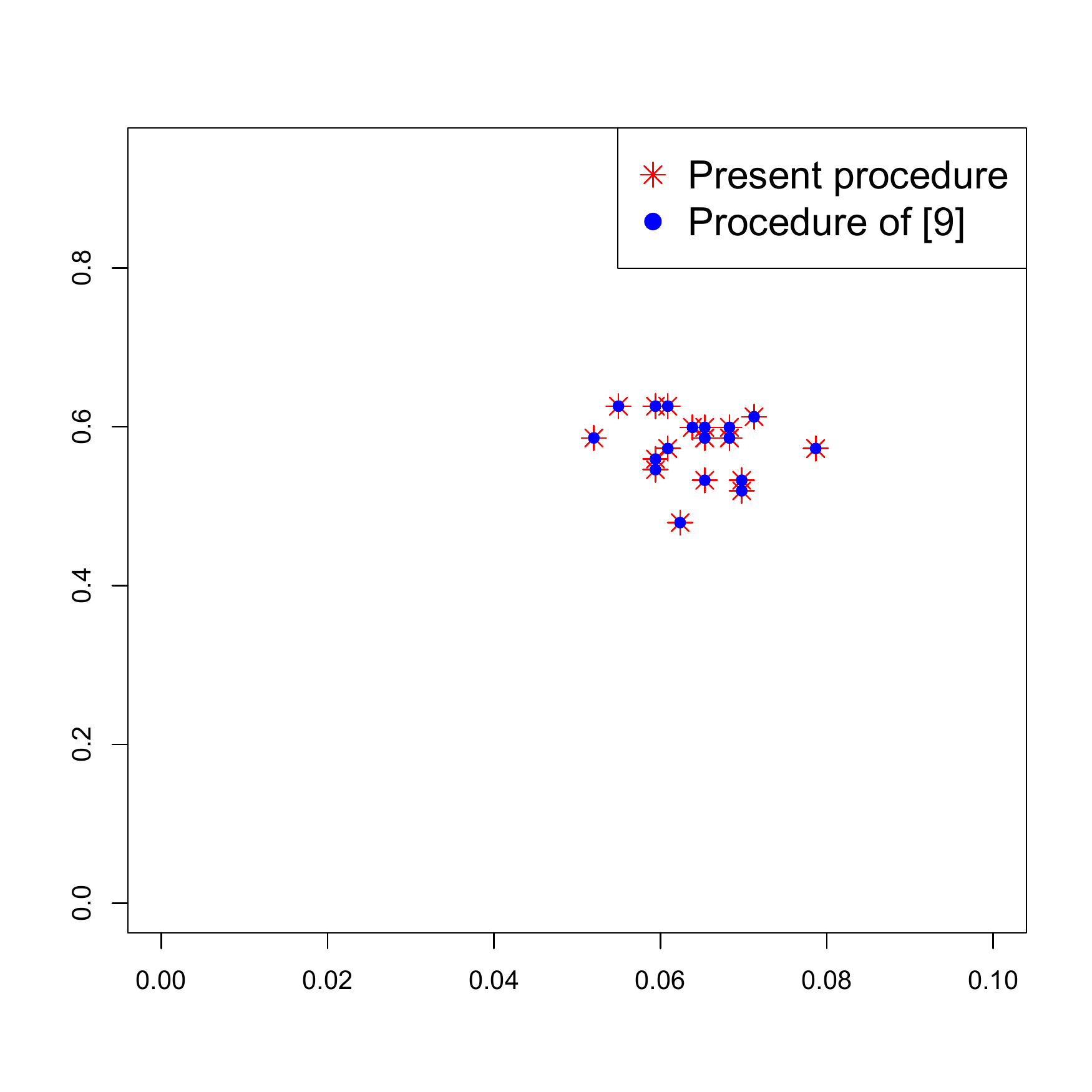}
\put(-130,3){{\small{\sf Type I error }}}
\put(-220,100){\rotatebox{90}{{\small{\sf Power}}}}
\vspace{0.5cm}
\label{fig:b5}
}\hspace{0.5cm}
\subfigure[$b = 100$] {
\includegraphics*[viewport = 15 0 480 455, width =
2.9in]{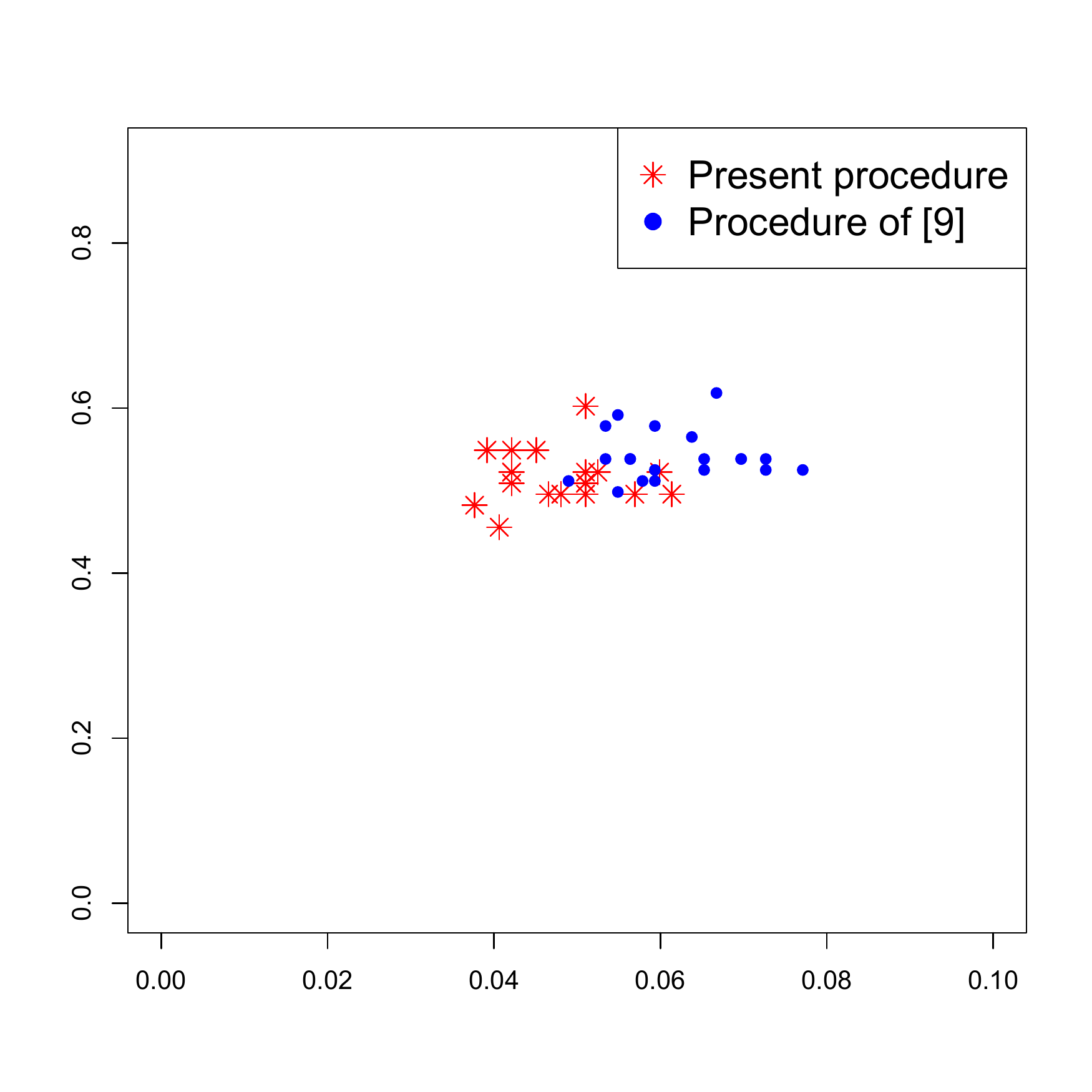}
\put(-130,3){{\small{\sf Type I error }}}
\put(-220,105){\rotatebox{90}{{\small{\sf Power}}}}
\label{fig:b100}
}
\caption{Comparison between testing procedure~\eqref{eq:decision-rule} and the one proposed in~\cite{confidenceJM} for one realization.}
\vspace{-0.5cm}
\end{figure}
\begin{figure}
\centering
\includegraphics*[viewport = 10 10 500 450, width = 3in]{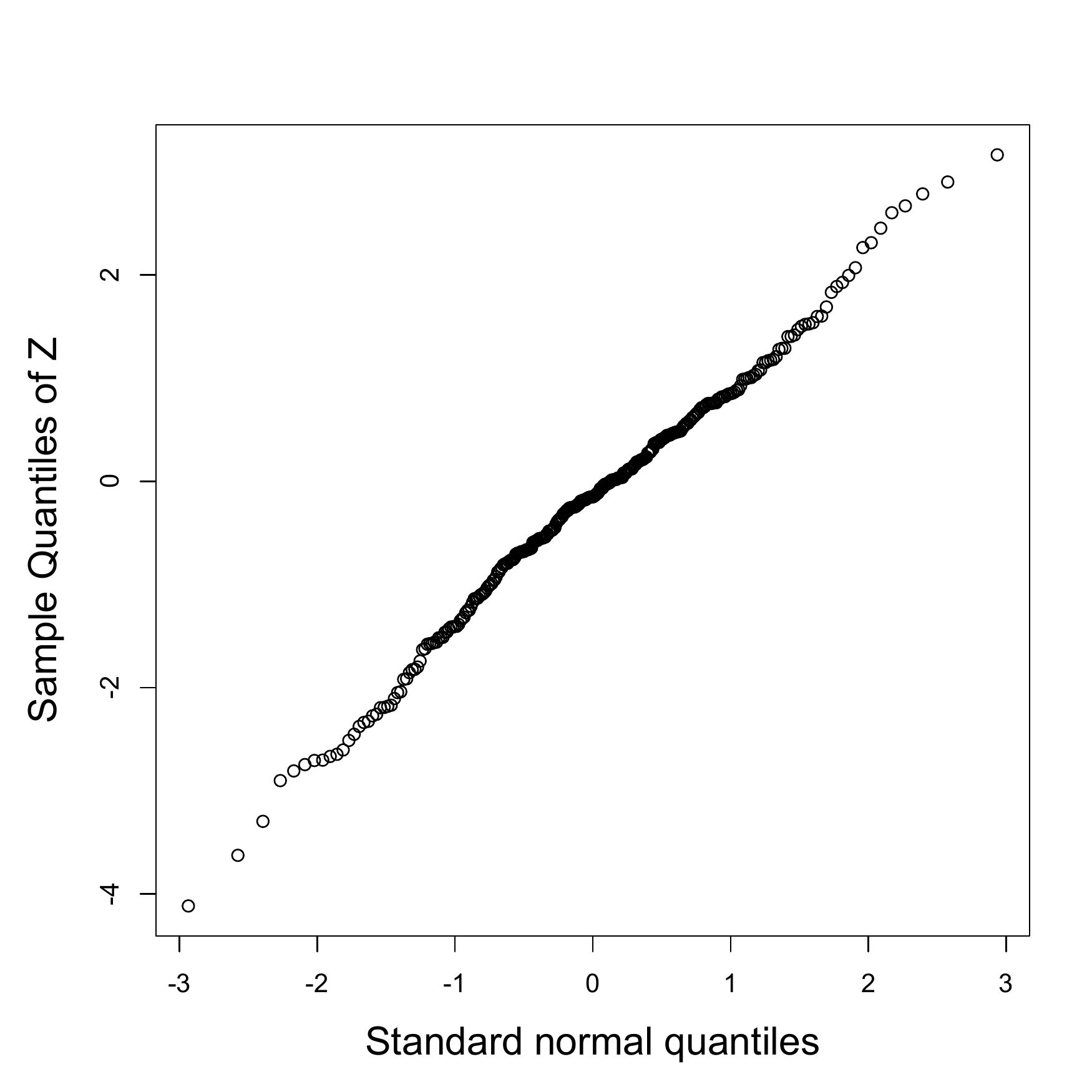}
\caption{Q-Q plot of $Z$ for one realization. Here $b = 75$.}\label{fig:qqplot}
\vspace{-0.5cm}
\end{figure}
We first bound $\|\Delta^{(1)}_A\|_2$ using Hoeffding's inequality.
%
Note that 
\vspace{-0.3cm}
$$\Delta^{(1)}_A = \sqrt{n}(\Omega\hSigma - \id)_{A,T}(\htheta-\theta_0)_T\,, $$
since $\supp(\htheta-\theta_0)\subseteq T$. Hence,

\begin{eqnarray}\label{eq:Delta1-1}
\|\Delta^{(1)}_A\|_2 \le \sqrt{n}\|(\Omega\hSigma - \id)_{A,T}\|_2 \|(\htheta-\theta_0)_T\|_2\,.
\end{eqnarray}
Let $R \equiv (\Omega\hSigma - \id)_{A,T}$ and define
$\cF_1\equiv\{u\in S^{p-1}:\,\supp(u)\subseteq[A] \}$, 
$\cF_2\equiv\{v\in S^{p-1}:\,\supp(v)\subseteq[T] \}$.
%
%
We have
\begin{align}
\|R\|_2 &= \underset{\substack{u,v\\ \|u\|, \|v\|\le 1}}{\sup} \<u,R v\> \nonumber \\
&=\underset{\substack{u,v\\ \|u\|, \|v\|\le 1}}{\sup} \Big(\<u,\frac{1}{n} \sum_{i=1}^n (\Omega X_i)_A (X_i^\sT)_T v\> - \<u_A,v_T\> \Big) \nonumber \\
&\le \underset{u\in \cF_1, v\in \cF_2}{\sup} \,\frac{1}{n} \sum_{i=1}^n
\Big( \<u, \Omega X_i\> \<X_i,v\> - \<u,v\>\Big)\,.\label{eq:R-1}
\end{align}

Fix $u\in \cF_1$ and $v \in \cF_2$. 
Let $\xi_i \equiv \<u, \Omega X_i\> \<X_i,v\> - \<u,v\>$.
The variables $\xi_{i}$ are independent and it is easy to see that $\E(\xi_i) = 0$. 
By~\cite[Remark 5.18]{Vershynin-CS}, we have
$$ \|\xi_i\|_{\psi_1} \le 2\|\<u, \Omega X_i\> \<X_i,v\>\|_{\psi_1}\,.$$
Moreover, by Lemma~\ref{lem:subG-subE}, 
\begin{align*}
\|\<u, \Omega X_i\> \<X_i,v\>\|_{\psi_1} &\le 2\|\<u, \Omega X_i\>\|_{\psi_2} \|\<X_i,v\>\|_{\psi_2}\\
 &= \|\Omega^{1/2}u\|_2 \|\Omega^{-1/2}v\|_2 \|\Omega^{1/2}X_i\|_{\psi_2}^2\\
 &\le \sqrt{\sigma_{\max}(\Sigma)/\sigma_{\min}(\Sigma)} \|\Omega^{1/2}X_i\|_{\psi_2}^2\,.
\end{align*}
Hence, $\max_{i\in[n]}\|\xi_i\|_{\psi_1} \le K$, for some constant $K$. Now, by applying Bernstein inequality for centered sub-exponential random
variables~\cite{Vershynin-CS}, for every $t\ge 0$, we have
\[
\prob\bigg(\frac{1}{n} \sum_{i=1}^n \xi_i  \ge t \Big) \le 2 \exp\bigg[-Cn\min\Big(\frac{t^2}{K^2}, \frac{t}{K}\Big) \bigg]\,,
\]
where $C>0$ is an absolute constant. Therefore, for any constant
$c_1>0$, since $n = \omega(s_0\log p)$, we have
\begin{eqnarray}\label{eq:R-2}
\prob\bigg(\frac{1}{n} \sum_{i=1}^n \xi_i \ge K\sqrt{\frac{c_1 s_0\log p}{Cn}} \bigg) \le  p^{-c_1s_0}\,.
\end{eqnarray}
In order to bound the right hand side of Eq.~\eqref{eq:R-1}, we use a \emph{$\eps$-net argument}.
Clearly, $\cF_1\cong S^{|A|-1}$ and $\cF_2 \cong S^{|T|-1}$ where $\cong$ denotes that the two objects are isometric.
By~\cite[Lemma 5.2]{Vershynin-CS}, there exists a $\frac{1}{2}$-net $\cN_1$ of $S^{|A|-1}$
(and hence of $\cF_1$) with size at most $5^{|A|}$. Similarly there exists a $\frac{1}{2}$-net $\cN_2$ of $\cF_2$
of size at most $5^{|T|}$. Hence, using Eq.~\eqref{eq:R-2} and taking union bound over all vectors in ${\cal N}_1$ and ${\cal N}_2$ , we obtain
\begin{align}
\underset{{u\in {\cal N}_1}, {v \in {\cal N}_2}}{\sup}  \frac{1}{n} \sum_{i=1}^n \<u, (\Omega X_i X_i^\sT - \id)v\> 
 \le K\sqrt{\frac{c_1 s_0\log p}{Cn}}\,, \label{eq:R-3}
\end{align}
with probability at least $1- 5^{|A|+|T|} p^{-c_1s_0}$.

The last part of the argument is based on the following lemma, whose proof is deferred to Appendix~\ref{app:net}.
\begin{lemma}\label{lem:net}
Let $M \in \reals^{p\times p}$. Then,
$$ \underset{u\in \cF_1, v\in \cF_2}{\sup} \<u,Mv\> \le 4 \underset{u\in {\cal N}_1, v\in {\cal N}_2}{\sup} \<u,Mv\>\,.$$
\end{lemma}
Employing Lemma~\ref{lem:net} and bound~\eqref{eq:R-3} in Eq.~\eqref{eq:R-1}, we arrive at
\begin{align}\label{eq:R-4}
\|R\|_2\le 4K\sqrt{\frac{c_1 s_0\log p}{Cn}}\,,
\end{align}
with probability at least $1- 5^{|A|+|T|} p^{-c_1s_0}$.

Finally, note that there are  less than $p^{c's_0}$ subsets $A,T$, with $|T|\le Cs_0$ and $|A|\le 2cs_0$,
for some constant $c'>0$.
Taking union bound over all these sets, we obtain that with high probability,
$$\|(\Omega\hSigma-\id)_{A,T}\|_2 \le  C\sqrt{s_0 \log p /n}\,,$$
for all such sets $A, T$, where $C = C(c,C_1,\sigma_{\max}(\Sigma),\sigma_{\min}(\Sigma))$
is a constant.

Now, plugging this bound and the bound~\eqref{eq:l2} (recalling that
$\kappa(s_0,q,3)$ is bounded away from zero with high probability by
\cite{rudelson2011reconstruction}
because $\sigma_{\rm
  min}(\Sigma)$ is bounded away from zero)  into Eq.~\eqref{eq:Delta1-1}, we get
\begin{eqnarray}\label{eq:Delta1}
\|\Delta^{(1)}_A\|_2 \le C \frac{\sigma s_0 \log p}{\sqrt{n}}\,,
\end{eqnarray}
with $C = C(c,C_1,\sigma_{\max}(\Sigma),\sigma_{\min}(\Sigma))$  a constant.

To bound $\|\Delta^{(2)}_A\|_2$, we bound each entry $\Delta^{(2)}_i$ separately:
\begin{align} \label{eq:Delta2i}
|\Delta^{(2)}_i| &\le \sqrt{n} \|\hOmega_{i,\cdot} - \Omega_{i,\cdot}\|_1  \|\hSigma(\theta_0 - \htheta)\|_\infty\,.
\end{align}
Note that the subgradient condition for optimization~\eqref{eq:LassoEstimator} reads
$$ \Sigma(\htheta - \theta_0) = \bX^\sT W/n + \lambda v(\htheta)\,,$$
with $v(\htheta)\in\partial\|\htheta\|_1$. Thus
$\|\Sigma(\htheta-\theta_0)\|_\infty = O(\sqrt{\log p/n})$, with high probability, for the choice of 
$\lambda = O(\sqrt{\log p /n})$. Therefore, Eq.~\eqref{eq:Delta2i} implies
\begin{eqnarray}\label{eq:Delta2}
\|\Delta^{(2)}\|_\infty = o(1)\,,
\end{eqnarray}
since by our assumption
$$\|\hOmega - \Omega\|_\infty = \max_{i\in [p]} \|\hOmega_{i,\cdot} - \Omega_{i,\cdot}\|_1 = o(1/\sqrt{\log p})\,.$$

We are now ready to bound $\|\Delta_A\|_2^2/|A|$. By triangle inequality,
\[
\|\Delta_A\|_2^2 \le 2\|\Delta^{(1)}_A\|_2^2 + 2\|\Delta^{(2)}\|_2^2\,.
\]
Applying bounds~\eqref{eq:Delta1} and~\eqref{eq:Delta2}, we obtain
\begin{eqnarray*}
\frac{\|\Delta_A\|_2^2}{|A|} \le C\frac{\sigma^2 s_0^2 (\log p)^2}{n |A|} + o(1)\,.
\end{eqnarray*}
This implies the thesis since$|A| \ge cs_0$ for some constant $c$. 
%

%
\subsection{Proof of Theorem~\ref{thm:error-power}}\label{proof:error-power}

We begin with a lemma that lower bounds the variance of $\htheta^u_i$.
\begin{lemma}\label{lem:BoundVariance}
Assume that the rows of $\bX\Omega^{1/2}$ are subgaussian, i.e.
$\|\Omega^{1/2} X_1\|_{\psi_2}<C$ for some constant $C$.  Further assume
that $1/C'\le \sigma_{\rm min}(\Sigma)\le \sigma_{\rm max}(\Sigma)\le C'$, for some constant $C'$. 
Finally assume that $\|\hOmega-\Omega\|_{\infty}\le 1/\sqrt{\log p}$ with probability
at least $1-\eps$.

Then for any constant $c_0 > 0$, there exist  
constants $c_1, c_2>0$ such that 
\begin{align}
\prob\Big(\max_{i\in [p]} \big|[\hOmega\hSigma\hOmega^\sT]_{ii} -\Omega_{ii}\big| \le
c_0\Big)\ge 1-c_1\, p^2\, e^{-c_2n}-\eps\, .
\end{align}
\end{lemma}
\begin{proof}
Fix $i\in [p]$, and let $v = \Omega^\sT e_i$ be the $i$-th column of
$\Omega$. Further let $\delta = (\Omega-\hOmega)^\sT e_i$. Then,
\begin{align*}
[\hOmega\hSigma\hOmega^\sT]_{ii} &= (v-\delta)^{\sT}\hSigma(v-\delta)
= v^{\sT}\hSigma v - 2 v^{\sT}\hSigma \delta +
\delta^{\sT}\hSigma\delta\,.
\end{align*}
Since $\hSigma \succeq 0$, we have 
$$v^\sT \hSigma \delta \le \sqrt{(v^\sT \hSigma v) (\delta^\sT \hSigma \delta)}.$$
Consequently,
$$\Big(\sqrt{v^{\sT}\hSigma v} - \sqrt{\delta^{\sT}\hSigma\delta}\Big)^2 
\le [\hOmega \hSigma \hOmega^\sT]_{ii} \le  \Big(\sqrt{v^{\sT}\hSigma v} + \sqrt{\delta^{\sT}\hSigma\delta}\Big)^2 \,.$$
Let $\cE$ be the event that $\|\hOmega-\Omega\|_{\infty}\le
1/\sqrt{\log p}$. On $\cE$, we have 
\begin{align*}
\delta^\sT \hSigma \delta  = \sum_{i,j \in [p]} \hSigma_{ij} \delta_i \delta_j \le |\hSigma|_\infty \|\delta\|_1^2 
\le |\hSigma|_\infty \|\Omega - \hOmega\|_\infty^2 \le \frac{|\hSigma|_\infty}{\log p}\,.
\end{align*}
It is therefore sufficient to prove that $|v^{\sT}\hSigma v - \Omega_{ii}| \le c_0/2$
with probability at least $1-c_1\,e^{-c_2
n}$, and $|\hSigma |_{\infty} <2$ with probability at least $1-c_1p^2\,e^{-c_2
n}$. The claim then follows by union bound.

Consider first $v^{\sT}\hSigma v$. We have $\E\{v^{\sT}\hSigma v\} =
[\Omega\Sigma\Omega]_{ii}=\Omega_{ii}$. Further
\begin{align}
v^{\sT}\hSigma v-\E(v^{\sT}\hSigma v) = \frac{1}{n}\sum_{j=1}^ne_i^{\sT}\Omega
\big[X_jX_j^{\sT}-\E(X_jX_j^{\sT})\big]\Omega^\sT e_i =\frac{1}{n}\sum_{j=1}^n\xi_j\, .
\end{align}
Here the $\xi_j$'s are i.i.d. sub-exponential with norm 
\begin{align}
\|\xi_1\|_{\psi_1}\le
4\|\<e_i,\Omega X_1\>\|_{\psi_2}^2\le 4C^2\|\Omega^{1/2} e_i\|_2^2\le
4C^2\sigma_{\rm max}(\Omega)\le 4C^2C'\,.
\end{align}
The claim then follows by applying Bernstein inequality for
sub-exponential random variables as in the proof in Section \ref{proof:main_lem}.

In order to bound $|\hSigma |_{\infty}$, we use
\begin{align}
\prob\big(|\hSigma |_{\infty}\ge 2\big) \le \sum_{i,j=1}^p 
\prob\big(|\hSigma_{ij} |\ge 2\big) \le \sum_{i,j=1}^p 
\prob\big(|\hSigma_{ij} -\E\hSigma_{ij}|\ge 1\big) \, ,
\end{align}
where the last inequality follows from $\E\hSigma_{ij} =
\Sigma_{ij}\le \sqrt{\Sigma_{ii}\Sigma_{jj}}=1$. Finally, the
probability of $|\hSigma_{ij} -\E\hSigma_{ij}|\ge 1$ is bounded once
again as above using Bernstein inequality.
\end{proof}

We next prove Eq.~\eqref{eq:typeI}. For $i\in [p]$, let 
$$A_i \equiv \frac{\sqrt{n}(\htheta_i^u-\theta_{0,i})}{\hsigma [\hOmega\hSigma\hOmega^\sT]_{ii}^{1/2}} = 
 \frac{\sigma}{\hsigma} \tZ_i + \frac{\Delta_i}{\hsigma [\hOmega \hSigma \hOmega^\sT]_{ii}^{1/2}}.$$
Invoking Proposition~\ref{pro:main_thm}, $\tZ_i|\bX \sim \normal(0,1)$.
For any constant $b \ge 0$, we have
\begin{align*}
\sum_{i\in S_0^c} \prob_{\theta_0}(|A_i| \ge b) &\le
\E\Big\{\sum_{i\in S_0^c}  \ind\Big(|\tZ_i| \ge  \frac{\hsigma}{\sigma} b - \frac{|\Delta_i|}{\sigma [\hOmega \hSigma \hOmega^\sT]_{ii}^{1/2}}\Big)\Big\}\\
&\le \E\Big\{\sum_{i\in S_0^c\backslash \cC_n(\eps)}  \ind\Big(|\tZ_i| \ge  \frac{\hsigma}{\sigma} b - \frac{|\Delta_i|}{\sigma [\hOmega \hSigma \hOmega^\sT]_{ii}^{1/2}}\Big) + |\cC_n(\eps)|  \Big\} \,.
\end{align*}
By definition, $|\Delta_i| \le \eps$ for $i\in S_0^c\backslash \cC_n(\eps)$. Hence,
\begin{align}\label{eq:T-1}
\frac{1}{p-s_0}\sum_{i\in S_0^c} \prob_{\theta_0}(|A_i| \ge b) \le
\frac{1}{p-s_0}\E\Big\{\sum_{i\in S_0^c}  \ind\Big(|\tZ_i| \ge  \frac{\hsigma}{\sigma} b - \frac{\eps}{\sigma [\hOmega \hSigma \hOmega^\sT]_{ii}^{1/2}}\Big) 
+ |\cC_n(\eps)| \Big\} \,.
\end{align}
Let $\cG \equiv \cG(\delta,c_0)$ be the following event:
$$ \cG \equiv \cG(\delta,c_0)  = \Big\{\max_{i\in [p]} \big|[\hOmega \hSigma \hOmega^\sT]_{ii} - \Omega_{ii} \big| \le c_0, \, |\hsigma / \sigma- 1| \le \delta \Big\}.$$
Then,
\begin{align}
\E\Big\{\sum_{i\in S_0^c}  \ind\Big(|\tZ_i| \ge  \frac{\hsigma}{\sigma} b - \frac{\eps}{\sigma [\hOmega \hSigma \hOmega^\sT]_{ii}^{1/2}}\Big) \Big\}
&\le \E\Big\{\Big[\sum_{i\in S_0^c}  \ind\Big(|\tZ_i| \ge  \frac{\hsigma}{\sigma} b - \frac{\eps}{\sigma [\hOmega \hSigma \hOmega^\sT]_{ii}^{1/2}}\Big)\Big] \cdot \ind(\cG) \Big\}
+ \prob(\cG^c) \nonumber\\
&\le \E\Big\{\Big[\sum_{i\in S_0^c}  \ind\Big(|\tZ_i| \ge  (1-\delta) b - \frac{\eps}{\sigma  \sqrt{\Omega_{ii}-c_0}}\Big) \Big]\cdot \ind(\cG) \Big\}
+ \prob(\cG^c) \nonumber\\
&\le \E\Big\{\sum_{i\in S_0^c}  \ind\Big(|\tZ_i| \ge  (1-\delta) b - \frac{\eps}{\sigma  \sqrt{\Omega_{ii}-c_0}}\Big) \Big\}
+ \prob(\cG^c)\,. \label{eq:Z-1}
\end{align}
Recalling that $\tZ_i|\bX \sim \normal(0,1)$, the following holds for any $b' \in \reals$:
\begin{align*}
\E(\ind(|\tZ_i|\ge b')) = \E\{\prob(|\tZ_i| \ge b'|\bX)\} = 2(1-\Phi(b'))\,.
\end{align*}
Plugging in Eq.~\eqref{eq:Z-1}, we obtain
\begin{align}
\underset{n\to \infty}{\lim\sup}\, \frac{1}{p-s_0}
\E\Big\{\sum_{i\in S_0^c}  \ind\Big(|\tZ_i| \ge  \frac{\hsigma}{\sigma} b - \frac{\eps}{\sigma [\hOmega \hSigma \hOmega^\sT]_{ii}^{1/2}}\Big) \Big\}
 \le 2\Big\{1-\Phi\Big((1-\delta) b - \frac{\eps}{\sigma \sqrt{\Omega_{ii} -c_0}} \Big) \Big\}\,, \label{eq:Z-2}
\end{align}
where in the last equality we used the fact that the event $\cG$ holds with high probability, i.e., $\lim_{n\to \infty} \prob(\cG^c) = 0$, as per Lemma~\ref{lem:BoundVariance}.

Employing bound~\eqref{eq:Z-2} in Eq.~\eqref{eq:T-1}, we get
\begin{align*}
\underset{n\to \infty}{\lim\sup}\, \frac{1}{p-s_0}\sum_{i\in S_0^c} \prob_{\theta_0}(|A_i| \ge b) &\le
2\Big\{1-\Phi\Big((1-\delta) b - \frac{\eps}{\sigma \sqrt{\Omega_{ii} -c_0}} \Big) \Big\}+ \lim_{n\to \infty}\E\Big(\frac{|\cC_n(\eps)|}{p-s_0}\Big)\\
&= 2\Big\{1-\Phi\Big((1-\delta) b - \frac{\eps}{\sigma \sqrt{\Omega_{ii} -c_0}} \Big) \Big\}\,,
\end{align*}
where the last step follows readily from Corollary~\ref{cor:main}. Since the above holds for all $\eps, \delta>0$, we obtain the following:
\begin{align}\label{eq:T-2}
\underset{n\to \infty}{\lim\sup}\, \frac{1}{p-s_0}\sum_{i\in S_0^c} \prob_{\theta_0}(|A_i| \ge b) \le
2(1-\Phi( b ))\,.
\end{align}

We are now ready to prove Eq.~\eqref{eq:typeI}. For the decision rule given in Eq.~\eqref{eq:decision-rule}, we have
\begin{align*}
&\underset{n\to \infty}{\lim\sup}\, \frac{1}{p-s_0} \sum_{i\in S_0^c} \prob_{\theta_0}(T_{i,\bX}(y)=1)
= \underset{n\to \infty}{\lim\sup}\, \frac{1}{p-s_0} \sum_{i\in S_0^c} \prob_{\theta_0}(P_i \le \alpha)\\
&= \underset{n\to \infty}{\lim\sup}\,  \frac{1}{p-s_0} \sum_{i\in S_0^c} \prob_{\theta_0}\Big(\Phi^{-1}(1-\frac{\alpha}{2}) \le |A_i| \Big) \le \alpha\,.
\end{align*}
Here, the second equality follows from construction of $p$-values $P_i$ as per Eq.\eqref{eq:p-value}, and
the fact that $\theta_{0,i} =0$, for $i\in S_0^c$; the inequality follows from Eq.~\eqref{eq:T-2}, with $b = \Phi^{-1}(1-\alpha/2)$.
%

Eq.~\eqref{eq:power} can be proved in a similar way, as follows: 
%
\begin{align*}
\frac{1}{s_0} \sum_{i\in S_0} \prob_{\theta_0}(T_{i,\bX}(y)=1) &=
\frac{1}{s_0} \sum_{i\in S_0} \prob_{\theta_0}(P_i \le \alpha) \\
&=  \frac{1}{s_0} \sum_{i\in S_0} \prob_{\theta_0} \Big(\Phi^{-1}(1-\frac{\alpha}{2}) \le \frac{\sqrt{n} |\htheta^u_i|}{\hsigma [\hOmega \hSigma \hOmega^\sT]_{ii}^{1/2}} \Big)  \\
&= \frac{1}{s_0} \sum_{i\in S_0} \prob_{\theta_0} \Big(\Phi^{-1}(1-\frac{\alpha}{2}) \le
\Big|\frac{\sigma}{\hsigma} \tZ_i + \frac{\sqrt{n} \theta_{0,i} + \Delta_i}
{\hsigma  [\hOmega \hSigma \hOmega^\sT]_{ii}^{1/2}}\Big| \Big)\,. 
\end{align*}
Define $\eta_i \equiv (\sqrt{n} \theta_{0,i} + \Delta_i)/(\sigma  [\hOmega \hSigma \hOmega^\sT]_{ii}^{1/2})$.
Rewriting the above identity we have
\begin{align}
\frac{1}{s_0} \sum_{i\in S_0} \prob_{\theta_0}(T_{i,\bX}(y)=1) &= 
\frac{1}{s_0}\E\Big\{ \sum_{i\in S_0} \ind \Big( \frac{\hsigma}{\sigma} \Phi^{-1}(1-\frac{\alpha}{2}) \le |\tZ_i + \eta_i | \Big) \Big\}\nonumber\\
&\ge \frac{1}{s_0}\E\Big\{ \sum_{i\in S_0\backslash\cC_n(\eps)} \ind \Big( \frac{\hsigma}{\sigma} \Phi^{-1}(1-\frac{\alpha}{2}) \le |\tZ_i + \eta_i | \Big)
\Big\}\,.\label{eq:eta-1}
\end{align}
By definition, $|\Delta_i| \le \eps$ for $i\in S_0\backslash\cC_n(\eps)$. Therefore, on the event $\cG$
we have 
$$|\eta_i|  \ge \eta^*_i \equiv \frac{\sqrt{n}|\theta_{0,i}|-\eps}{\sigma \sqrt{\Omega_{ii}+c_0}}\,, \quad \text{for }i \in S_0\backslash \cC_n(\eps)\,.$$
 Moreover, $\hsigma/\sigma \le 1+\delta$.
Fix arbitrary $\delta' >0$ and define the event $\tilde{\cG}$ as in the following
\[
\tilde{\cG} \equiv \cG \cap \Big\{\frac{\cC_n(\eps)}{|s_0|} \le \delta' \Big\}\,.
\]
Using Eq.~\eqref{eq:eta-1}, we have
\begin{align}
\frac{1}{s_0} \sum_{i\in S_0} \prob_{\theta_0}(T_{i,\bX}(y)=1)
&\ge \frac{1}{s_0}\E\Big\{ \Big[ \sum_{i\in S_0\backslash\cC_n(\eps)} \ind \Big(\frac{\hsigma}{\sigma} \Phi^{-1}(1-\frac{\alpha}{2}) \le |\tZ_i + \eta_i | \Big)\Big]
\cdot \ind(\tilde{\cG})
\Big\} - \prob(\tilde{\cG}^c) \nonumber\\
&\ge \frac{1}{s_0}\E\Big\{ \Big[\sum_{i\in S_0} \ind \Big(\frac{\hsigma}{\sigma} \Phi^{-1}(1-\frac{\alpha}{2}) \le |\tZ_i + \eta_i | \Big)
- |\cC_n(\eps)| \Big] \cdot \ind(\tilde{\cG})\Big\} - \prob(\tilde{\cG}^c) \nonumber\\
&\ge \frac{1}{s_0}\E\Big\{\sum_{i\in S_0} \ind \Big((1+\delta) \Phi^{-1}(1-\frac{\alpha}{2}) \le |\tZ_i + \eta^*_i | \Big)
\Big\} - \delta' - \prob(\tilde{\cG}^c)\,,
\end{align}
where the last step follows from definition of event $\tilde{\cG}$.
Hence,
\begin{align*}
 &\underset{n\to \infty}{\lim\inf}\, \frac{1}{s_0(1-\beta^*(\theta_0;n))} \Big\{\sum_{i\in S_0} \prob_{\theta_0}(T_{i,\bX}(y)=1) \Big\} \\
 &\ge \underset{n\to \infty}{\lim\inf}\, \frac{1}{s_0(1-\beta^*(\theta_0;n))}\, \E\Big\{ \sum_{i\in S_0} \ind \Big((1+\delta) \Phi^{-1}(1-\frac{\alpha}{2}) \le |\tZ_i + \eta^*_i | \Big)\Big\}
  - \delta' - \lim_{n\to \infty} \prob(\tilde{\cG}^c)\,. 
\end{align*}
Given that $\hsigma$ is a consistent estimator for $\sigma$, and using Lemma~\ref{lem:BoundVariance} and Corollary~\ref{cor:main}, the event $\tilde{\cG}$ holds with high probability, i.e., $\lim_{n \to \infty} \prob(\tilde{\cG}^c) = 0$. Since the above bound holds for all $\delta', \eps, c_0 > 0$, we get
\begin{align*}
 &\underset{n\to \infty}{\lim\inf}\, \frac{1}{s_0(1-\beta^*(\theta_0;n))} \Big\{\sum_{i\in S_0} \prob_{\theta_0}(T_{i,\bX}(y)=1) \Big\} \\
 &\ge \underset{n\to \infty}{\lim\inf}\, \frac{1}{s_0(1-\beta^*(\theta_0;n))} \E\Big\{ \sum_{i\in S_0} \ind \Big(\Phi^{-1}(1-\frac{\alpha}{2}) \le \Big|\tZ_i + \frac{\sqrt{n}|\theta_{0,i}|}{\sigma \sqrt{\Omega_{ii}}} \Big| \Big)\Big\}\\
& = \underset{n\to \infty}{\lim\inf}\, \frac{1}{(1-\beta^*(\theta_0;n))} \Big\{ \frac{1}{s_0} \sum_{i\in S_0} G\Big(\alpha, \frac{\sqrt{n}|\theta_{0,i}|}{\sigma \sqrt{\Omega_{ii}}}\Big)\Big\} = 1\,.
\end{align*}
The last step follows from definition of function $G(\cdot,\cdot)$, as per Eq.~\eqref{Eq:Gfun}, and the fact that $\tZ_i|\bX \sim \normal(0,1)$.
%
%
%
\section*{Acknowledgements}
A.J. is supported by a Caroline and Fabian
Pease Stanford Graduate Fellowship.
This work was partially supported by the NSF CAREER award CCF-0743978, the NSF
grant DMS-0806211, and the grants AFOSR/DARPA FA9550-12-1-0411 and FA9550-13-1-0036.
\appendix
\section{Proof of Proposition~\ref{pro:main_thm}}
\label{app:main_thm}

Plugging in $Y= \bX \theta_0 +W$, we have
\begin{align*}
&\sqrt{n}(\htheta^u - \theta_0)\\ 
&= \sqrt{n} \Big\{\htheta - \theta_0 + \frac{1}{n} \hOmega \bX^\sT W + 
\frac{1}{n} \hOmega \bX^\sT \bX(\theta_0 - \htheta) \Big\}\\
&= Z+\Delta\,,
\end{align*}
where $Z = \hOmega \bX^\sT W/\sqrt{n}$, and $\Delta = \sqrt{n} (\hOmega\hSigma - \id)(\theta_0 -\htheta)$.
Conditional on $\bX$, we have $Z \sim \normal(0,\sigma^2 \hOmega \hSigma \hOmega^\sT)$, since $W \sim \normal(0,\sigma^2\id)$.
\section{Proof of Proposition~\ref{pro:Bickel}}
\label{app:Lasso-supp-size}
This proposition is a slightly improved version of Theorem 7.2 in~\cite{BickelEtAl}, in that we 
replace $\phi_{\max}(p)$ by $\phi_{\max}(n)$ in the bound on $\|\htheta\|_0$.
Here, we prove Eq.~\eqref{eq:Lasso-supp-size}.  

Let $\hS \equiv \supp(\htheta)$. Recall that the stationarity condition for the Lasso cost function reads
$\bX^{\sT}(y-\bX\htheta) = n\lambda\, v(\htheta)$, where
$v(\htheta)\in\partial\|\htheta\|_1$. Equivalently,
\begin{align*}
\frac{1}{n}\bX^{\sT}\bX(\theta_0-\htheta) = \lambda\,
v(\htheta)-\frac{1}{n}\bX^{\sT} w\, .
\end{align*}
As proved in \cite{BickelEtAl}, 
$\|\bX^\sT w\|_\infty \le n\lambda/2$ with high probability. Thus for all $i\in \hS$
\begin{align*}
\left|\frac{1}{n}\bX^{\sT}\bX(\theta_0-\htheta)\right|_i \ge
\frac{\lambda}{2}\, .
\end{align*}

Let $\proj_{\hS}$ be the orthogonal projector in $\reals^p$ on the
subspace of vectors with support in $\hS$. Squaring and summing the last identity
over $i\in\hS$, we obtain, for $h \equiv
n^{-1/2}\bX(\theta_0-\htheta)$,
\begin{align*}
|\hS|&\le \frac{4}{\lambda^2}\,\<h,\frac{1}{n}\bX\proj_{\hS}\bX^{\sT}\,
h\>\\
&\le  \frac{4}{\lambda^2}\,\phi_{\rm max}(|\hS|)^2\|h\|_2^2
\le \frac{4\phi_{\rm max}(n)^2}{\lambda^2} \|h\|^2\, ,
\end{align*}
where the last inequality follows because $|\hS|\le n$ by the fact
that the columns of $\bX$ are in generic positions (see e.g.~\cite[Lemma 3]{Tib-Unique}). By  
\cite[Theorem 6.2]{BickelEtAl}, we have
$n^{-1}\|\bX(\theta_0-\htheta)\|_2^2\le
16\lambda^2s_0/\kappa(s_0,3)^2$, whence the claim follows.
 
\section{Auxiliary lemmas}
\begin{lemma}\label{lem:subG-subE}
For any two random variables $X$ and $Y$, we have
\[
\|XY\|_{\psi_1} \le 2\|X\|_{\psi_2} \|Y\|_{\psi_2}\,.
\]
\end{lemma}
\begin{proof}
By definition of sub-exponential and sub-gaussian norms, we write
\begin{align*}
\|XY\|_{\psi_1} &= \underset{p\ge 1}{\sup} \, p^{-1} \E(|XY|^p)^{1/p}\\
&\le \underset{p\ge 1}{\sup} \, p^{-1} \E(|X|^{2p})^{1/2p} \E(|Y|^{2p})^{1/2p}\\
&\le 2 \Big( \underset{q\ge 2}{\sup} \,q^{-1/2} \E(|X|^q)^{1/q} \Big) \Big(  \underset{q\ge 2}{\sup} \,q^{-1/2} \E(|Y|^q)^{1/q} \Big)\\
&\le 2 \|X\|_{\psi_2} \|Y\|_{\psi_2}\,.
\end{align*}
Here, the first inequality follows from Cauchy-Schwartz inequality. 
\end{proof}
\section{Proof of Lemma~\ref{lem:net}}
\label{app:net}
Each vector $u\in \cF_1$ can be written as $u = \sum_{i=0}^\infty 2^{-i} u_i$,
where $u_i \in {\cal N}_1$. Similarly, each vector $v\in \cF_2$ can be written
as $v = \sum_{j=0}^\infty 2^{-j} v_j$. Therefore,
\[
\<u,Mv\> = \sum_{i,j=0}^\infty 2^{-i-j} \<u_i,Mv_j\> \le 4 \underset{u'\in {\cal N}_1,v'\in {\cal N}_2}{\sup} \<u',Mv'\>\,.
\]
The result follows.
%



\bibliographystyle{amsalpha}
\bibliography{all-bibliography}

\providecommand{\bysame}{\leavevmode\hbox to3em{\hrulefill}\thinspace}
\providecommand{\MR}{\relax\ifhmode\unskip\space\fi MR }
\providecommand{\MRhref}[2]{%
  \href{http://www.ams.org/mathscinet-getitem?mr=#1}{#2}
}
\providecommand{\href}[2]{#2}
\begin{thebibliography}{vdGBR13}

\bibitem[BCW11]{belloni2011square}
Alexandre Belloni, Victor Chernozhukov, and Lie Wang, \emph{Square-root lasso:
  pivotal recovery of sparse signals via conic programming}, Biometrika
  \textbf{98} (2011), no.~4, 791--806.

\bibitem[BRT09]{BickelEtAl}
P.~J. Bickel, Y.~Ritov, and A.~B. Tsybakov, \emph{{Simultaneous analysis of
  Lasso and Dantzig selector}}, Annals of Statistics \textbf{37} (2009),
  1705--1732.

\bibitem[B{\"u}h12]{BuhlmannSignificance}
P.~B{\"u}hlmann, \emph{{Statistical significance in high-dimensional linear
  models}}, {\sf arXiv:1202.1377}, 2012.

\bibitem[BvdG11]{buhlmann2011statistics}
Peter B{\"u}hlmann and Sara van~de Geer, \emph{Statistics for high-dimensional
  data}, Springer-Verlag, 2011.

\bibitem[CD95]{BP95}
S.S. Chen and D.L. Donoho, \emph{{Examples of basis pursuit}}, Proceedings of
  Wavelet Applications in Signal and Image Processing III (San Diego, CA),
  1995.

\bibitem[CT05]{CandesTao}
E.~J. Cand\'es and T.~Tao, \emph{{Decoding by linear programming}}, IEEE Trans.
  on Inform. Theory \textbf{51} (2005), 4203--4215.

\bibitem[CT07]{Dantzig}
E.~Cand\'es and T.~Tao, \emph{{The Dantzig selector: statistical estimation
  when p is much larger than n}}, Annals of Statistics \textbf{35} (2007),
  2313--2351.

\bibitem[DMM09]{DMM09}
D.~L. Donoho, A.~Maleki, and A.~Montanari, \emph{{Message Passing Algorithms
  for Compressed Sensing}}, Proceedings of the National Academy of Sciences
  \textbf{106} (2009), 18914--18919.

\bibitem[DT05]{DoTa05}
D.~L. Donoho and J.~Tanner, \emph{Neighborliness of randomly-projected
  simplices in high dimensions}, Proceedings of the National Academy of
  Sciences \textbf{102} (2005), no.~27, 9452--9457.

\bibitem[JM13a]{confidenceJM}
Adel Javanmard and Andrea Montanari, \emph{{Confidence Intervals and Hypothesis
  Testing for High-Dimensional Regression}}, {\sf arXiv:1306.3171}, 2013.

\bibitem[JM13b]{javanmard2013hypothesis}
\bysame, \emph{Hypothesis testing in high-dimensional regression under the
  gaussian random design model: Asymptotic theory}, {\sf arXiv:1301.4240},
  2013.

\bibitem[MB06]{MeinshausenBuhlmann}
N.~Meinshausen and P.~B{\"u}hlmann, \emph{High-dimensional graphs and variable
  selection with the lasso}, Ann.~Statist. \textbf{34} (2006), 1436--1462.

\bibitem[RZ13]{rudelson2011reconstruction}
Mark Rudelson and Shuheng Zhou, \emph{Reconstruction from anisotropic random
  measurements}, IEEE Transactions on Information Theory \textbf{59} (2013),
  no.~6, 3434--3447.

\bibitem[SZ12]{SZ-scaledLasso}
Tingni Sun and Cun-Hui Zhang, \emph{Scaled sparse linear regression},
  Biometrika \textbf{99} (2012), no.~4, 879--898.

\bibitem[Tib96]{Tibs96}
R.~Tibshirani, \emph{{Regression shrinkage and selection with the Lasso}}, J.
  Royal. Statist. Soc B \textbf{58} (1996), 267--288.

\bibitem[Tib13]{Tib-Unique}
Ryan~J. Tibshirani, \emph{{The lasso problem and uniqueness}}, Electronic
  Journal of Statistics \textbf{7} (2013), 1456--1490.

\bibitem[vdGBR13]{GBR-hypothesis}
S.~van~de Geer, P.~B{\"u}hlmann, and Y.~Ritov, \emph{{On asymptotically optimal
  confidence regions and tests for high-dimensional models}}, {\sf
  arXiv:1303.0518}, 2013.

\bibitem[Ver12]{Vershynin-CS}
R.~Vershynin, \emph{Introduction to the non-asymptotic analysis of random
  matrices}, Compressed Sensing: Theory and Applications (Y.C. Eldar and
  G.~Kutyniok, eds.), Cambridge University Press, 2012, pp.~210--268.

\bibitem[Wai09]{Wainwright2009LASSO}
M.J. Wainwright, \emph{Sharp thresholds for high-dimensional and noisy sparsity
  recovery using $\ell_1$-constrained quadratic programming}, IEEE Trans. on
  Inform. Theory \textbf{55} (2009), 2183--2202.

\bibitem[ZZ11]{ZhangZhangSignificance}
C.-H. Zhang and S.S. Zhang, \emph{{Confidence Intervals for Low-Dimensional
  Parameters in High-Dimensional Linear Models}}, {\sf arXiv:1110.2563}, 2011.

\end{thebibliography}

\end{document}